\documentclass[11pt,a4paper]{article}
\usepackage[utf8]{inputenc}
\usepackage{amsmath}
\usepackage{amsthm}
\usepackage{amsfonts}
\usepackage{amssymb}
\usepackage{graphicx}
\usepackage{subfigure}
\usepackage{tikz}
\usetikzlibrary{calc,arrows,babel,positioning,shapes,patterns,decorations.markings}

\newtheorem{definition}{Definition}
\newtheorem{theorem}{Theorem}
\newtheorem{lemma}[theorem]{Lemma}

\newtheorem{corollary}[theorem]{Corollary}

\usepackage{authblk}

\providecommand{\keywords}[1]
{
  \small	
  \textbf{\textit{Keywords:}} #1
}

\title{Euler dynamic H-trails in edge-colored graphs\thanks{This research was supported by grants CONACYT FORDECYT-PRONACES/ 39570/2020 and UNAM DGAPA-PAPIIT IN102320. The second author is a doctoral student from Posgrado en Ciencias Matemáticas, Universidad Nacional Autónoma de México (UNAM) and received a fellowship 782239 from CONACYT.} }
\author[1]{Galeana-Sánchez Hortensia}
\author[1]{Vilchis-Alfaro Carlos\thanks{Corresponding author. Email: vilchiscarlos@ciencias.unam.mx}}
\affil[1]{\normalsize Instituto de Matem\'{a}ticas, Universidad Nacional Aut\'{o}noma de M\'{e}xico, Área de la investigación científica, Circuito Exterior, Ciudad Universitaria, 04510 Coyoacán, CDMX, México}
\date{}

\usepackage[left=23mm,right=23mm,top=25mm,bottom=25mm]{geometry}

\begin{document}
\maketitle

\begin{abstract}
Alternating Euler trails has been extensively studied for its diverse practical and theorical applications. Let H be a graph possibly with loops and G a graph without loops.  In this paper we deal with any fixed coloration of E(G) with V(H) (H-coloring of G). A sequence $W=(v_0,e_0^1,\ldots,e_0^{k_0},v_1,e_1^1,\ldots,e_{n-1}^{k_{n-1}},v_n)$ in $G$, where for each $i\in\{0,\ldots,n-1\}$, $k_i\geq 1$ and $e_i^j=v_iv_{i+1}$ is an edge in $G$, for every $j\in\{1,\ldots, k_i\}$, is a dynamic $H$-trail if $W$ does not repeat edges and $c(e_i^{k_i})c(e_{i+1}^1)$ is an edge in $H$, for each $i\in\{0,\ldots,n-2\}$. In particular a dynamic $H$-trail is an alternating Euler trail when $H$ is a complete graph without loops and $k_i=1$, for every $i\in\{1,\ldots,n-1\}$.

In this paper, we introduce the concept of dynamic $H$-trail, which arises in a natural way in the modeling of many practical problems, in particular, in theoretical computer science. 

We provide necessary and sufficient conditions for the existence of closed Euler dynamic $H$-trail in $H$-colored multigraphs. Also we provide polynomial time algorithms that allows us to convert a cycle in an auxiliary graph, $L_2^H(G)$, in a closed dynamic H-trail in $G$, and vice versa, where $L_2^H(G)$ is a non-colored simple graph obtained from $G$ in a polynomial time.
\end{abstract}

\section{Introduction}
For basic concepts, terminology and notation not defined here, we refer the reader to \cite{bang} and \cite{bondy}. Throughout this work, we will consider graphs, multigraphs (graphs allowing parallel edges) and simple graphs (graphs with no parallel edges nor loops). Let $G$ be a multigraph, $V(G)$ and $E(G)$ will denoted the sets of vertices and edges of $G$, respectively.

A \textit{spanning circuit} in a graph $G$ is defined as a closed trail that contains each vertex of $G$. We will say that a graph is \textit{supereulerian} if it contains a spanning circuit. Let $T$ be a spanning circuit in $G$: if $T$ visits each vertex of $G$ exactly once, then $T$ will be called \textit{Hamilton cycle}; and if $T$ visits each edge of $G$, then $T$ will be called \textit{Euler trail}. We will say that a graph is \textit{hamiltonian} if it has a Hamilton cycle and  \textit{eulerian} if it contains an Euler trail.

In \cite{supereulerian}, Harary and Nash-Williams introduced the following graphs. Let $G$ be a graph with $p$ vertices and $q$ edges, for $n \geq 2$, $L_n(G)$ will denote the graph with $nq$ vertices that are obtained as follows: for each edge $e=uv$ of $G$, we take two vertices $f(u,e)$ and $f(v,e)$ in $L_n(G)$ and adding a path with $n-2$ new intermediate vertices connecting $f(u,e)$ and $f(v,e)$; finally, for each vertex $u$ of $G$, we add an edge joining $f(u,e)$ and $f(u,g)$, whenever $e$ and $g$ are distinct edges with end point $u$. They also proved the followings relationships between eulerian graphs and hamiltonian cycles of $L_n(G)$.

\begin{theorem}[\cite{supereulerian}]
Let $G$ be a graph. The following assertions holds:
\begin{enumerate}
\item If $G$ is eulerian, then $L_n(G)$ is hamiltonian, for every $n \geq 2$.
\item If $L_n(G)$ is hamiltonian, for some $n \geq 3$, then $G$ is eulerian.
\item $G$ is superulerian if and only if $L_2(G)$ is hamiltonian.
\end{enumerate}
\end{theorem}

An \textit{edge-coloring} of a graph $G$ is defined as a mapping $c : E(G) \rightarrow \mathbb{N}$, where $\mathbb{N}$ is the set of natural numbers. We will say that $G$ is an \textit{edge-colored graph} whenever we are taking a fixed edge-coloring of $G$. A \textit{properly colored spanning circuit} is a spanning circuit with no consecutive edges having the same color, including the first and the last in a closed spanning circuit. An edge-colored graph is \textit{supereulerian}, if $G$ contains a properly colored spanning circuit. Several authors have worked with this concept, for example, Bang-Jensen, Bellitto and Yeo \cite{PC-supereulerian}; Guo, Li, Li and Zhang \cite{PC-supereulerian2}; Guo, Broersma, Li and Zhang \cite{PC-supereulerian3}. Properly colored walks are of interest as a generalization of walk in undirected and directed graphs, see \cite{bang}, as well as, in graph theory application, for example, in genetic and molecular biology \cite{dorninger,dorninger2,a4}, social science \cite{cs}, channel assignment is wireless networks \cite{raw,raw2}. For more details on properly colored walks, we refer the reader to Chapter 16 of \cite{bang}.

Kotzig \cite{kotzig} gave the following characterization of edge-colored multigraphs containing properly colored closed Euler trail.

\begin{theorem}[Kotzig \cite{kotzig}]\label{theorem:Kotzig}
Let $G$ be an edge-colored eulerian multigraph. Then $G$ has a properly colored closed Euler trail if and only if $\delta_i(x) \leq \sum_{j\neq i} \delta_j (x)$, where $\delta_i(x)$ is the number of edges with color $i$ incident with $x$, for each vertex $x$ of $G$.
\end{theorem}

Different kinds of edge-coloring in undirected and directed graphs have been studied, for example, in \cite{Linek} the arcs of a tournament were colored with the vertices of a poset. In this paper we will consider the following edge-coloring. Let $H$ be a graph possibly with loops and $G$ be a graph without loops. An \textit{$H$-coloring} of $G$ is a function $c: E(G) \rightarrow V(H)$. We will say that $G$ is an $H$\textit{-colored graph}, whenever we are taking a fixed $H$-coloring of $G$. A walk $W=(v_0,e_0,v_1,e_1,\ldots,e_{k-1},v_k)$ in $G$, where $e_i=v_i v_{i+1}$ for every $i$ in $\{0,\ldots,k-1\}$, is an \textit{$H$-walk} if $(c(e_0),a_0,c(e_1),\ldots,c(e_{k-2}),a_{k-2},c(e_{k-1}))$ is a walk in $H$, with $a_i=c(e_i)c(e_{i+1})$ for every $i \in \{0,\dots,k-2\}$. We will say that $W$ is \textit{closed} if $v_0=v_k$ and $c(e_{k-2})c(e_0) \in E(H)$. Let $W$ be an $H$-walk, if $W$ is a trail then $W$ will be called \textit{$H$-trail}. Notice that if $H$ is a complete graph without loops, then an $H$-walk is a properly colored walk.

Since, properly colored walks have been very useful for modeling and solving different problems, as we mentioned above, $H$-colored graphs and $H$-walks can also model interesting problems, for example, routing problems or in communications networks were some transitions are restricted because of natural phenomenon, external attack or failure. More applications of colored walks with restrictions in color's succession can be found in \cite{a3}.

The concepts of $H$-coloring and $H$-walks were introduced, for the first time by Linek and Sands in \cite{Linek}, and have been worked mainly in the context of kernel theory and related topics, see \cite{H-caminos1,H-caminos2,H-caminos3}. 

In \cite{H-paseos}, Galeana-S\'{a}nchez, Rojas-Monroy, S\'{a}nchez-L\'{o}pez and Villarreal-Vald\'{e}s defined the graph $G_u$ as follows: Let $u$ be a vertex of $G$; $G_u$ is the graph such that $V(G_u)=\{e \in E(G): e \textrm{ is incident}$ $\textrm{with } u\}$, and two different vertices $a$ and $b$ are joining by only one edge in $G_u$ if and only if $c(a)$ and $c(b)$ are adjacent in $H$.

They also showed necessary and sufficient conditions for the existence of closed Euler $H$-trails, as follows.

\begin{theorem}[\cite{H-paseos}]\label{theorem:H-trails}
Let $H$ be a graph possibly with loops and $G$ be an $H$-colored multigraph without loops. Suppose that $G$ is Eulerian and $G_u$ is a complete $k_u$-partite graph, for every $u$ in $V(G)$ and for some $k_u$ in $\mathbb{N}$. Then $G$ has a closed Euler $H$-trail if and only if $\vert C_i^u \vert \leq \sum_{j \neq i} \vert C_j^u \vert$ for every $u$ in $V(G)$, where $\{C_1^u,...,C_{k_u}^u\}$ is the partition of $V(G_u)$ into independent sets.
\end{theorem}

In \cite{H-dinamicos}, Ben\'{i}tez-Bobadilla, Galeana-S\'{a}nchez and Hern\'{a}ndez-Cruz introduced a generalization of $H$-walk as follows. They allowed ``lane changes", i.e., they allowed concatenation of two $H$-walks as long as, the last edge of the first one and the first edge of the second one satisfy that $x_{k-1}=y_0$ and $x_k=y_1$. As a result, they defined the following new concept: Let $G$ be an $H$-colored multigraph, a \textit{dynamic $H$-walk} in $G$ is a sequence of vertices $W=(x_0,x_1,\ldots,x_k)$ in $G$ such that for each $i \in \{0,\ldots,k-2\}$ there exists an edge $f_i=x_ix_{i+1}$ and there exists an edge $f_{i+1}=x_{i+1}x_{i+2}$ such that $c(f_i)c(f_{i+1})$ is an edge in $H$.

For the purpose of this paper, we need a definition and notation that will allows us to know the edges that belong to a dynamic $H$-walk. So, we will say that $W=(v_0,e_0^1,\ldots, e_0^{k_0},v_1,e_1^1,\ldots,$ $e_1^{k_1},v_2,\ldots,v_{n-1},e_{n-1}^{k_{n-1}},\ldots,$ $e_{n-1}^{k_{n-1}},v_n)$, where for each $i \in \{0,\ldots, n-1\}$, $k_i \geq 1$ and $e_i^j = v_iv_{i+1}$ for every $j \in \{1,\ldots, k_i \}$, is a \textit{dynamic $H$-walk} if $c(e_i^{k_i})c(e_{i+1}^1)$ is an edge in $H$, for each $i \in \{0,\ldots,n-2\}$. If $W$ is a dynamic $H$-walk that does not repeat edges, then $W$ will be called \textit{dynamic} $H$\textit{-trail}. We will say that a dynamic $H$-trail, $W$, is an \textit{Euler dynamic $H$-trail} if $E(G)=E(W)$. We will say that a dynamic $H$-trail is a \textit{closed dynamic $H$-trail} whenever a) $v_0=v_n$ and $c(e_{n-1}^{k_{n-1}})c(e_0^1)$ is an edge in $H$; or b) $v_1=v_n$ and $e_{n-1}^{k_{n-1}}$ and $e_0^1$ are parallel in $G$. Notice that if $W$ is a closed dynamic $H$-walk that satisfies condition b, then $W$ can be rewrite as $W=(v_1,e_1^1,...,v_{n-1}=v_0,e_{n-1}^1,\ldots,e_{n-1}^{k_{n-1}},e_0^1,\ldots, e_0^{k_0},v_n=v_1)$, and $W$ satisfies condition (a) (unless $n=1$, i.e., $W$ is of the form $(x,e_0,\ldots,e_k,y)$, where $k \geq 1$).

It follows from the definition of dynamic $H$-trail that every $H$-trail in $G$ is a dynamic $H$-trail in $G$. Moreover, if $G$ has no parallel edges, then every dynamic $H$-trail is an $H$-trail.

A motivation for the study of dynamic $H$-walks are their possibles applications. For example, suppose that it is require to send a message from point $A$ to point $B$ through communication network which can have faults in its links (such as damage, attack, virus or blockage), where each one of its edge has an assigned color depending on the probability that the message is delivered correctly (in time, complete, unchanged or virus-free). In this case, the vertices of $H$ are the colors assigned to the edges of the network and the set of edges of $H$ will depend on the transitions that are convenient (for example, if transition with the same probability of failure are forbidden, then $H$ will have no loops). In practice, it is possible to send messages simultaneously over two or more connections. In these cases, the use of dynamic H-walks can improve message delivery.

Motivated by the Theorem \ref{theorem:H-trails} and its proof, we think that it is possible to give conditions, similar to those of Theorem \ref{theorem:H-trails}, on an $H$-colored multigraph that guarantee the existence of closed Euler dynamic $H$-trail. In the process, we find an auxiliary graph, we call it $L_3^H(G)$, that allows us to link the closed dynamic $H$-trails with hamiltonian graphs.

The rest of the paper is organized as follows. Section \ref{Section:NyT} is devoted to give some notation and terminology, which will be used throughout the paper. In Section \ref{Section:Main}, we will define the graph $L_n^H(G)$. Then we will prove that $G$ contains a closed Euler dynamic $H$-trails if and only if $L_n^H(G)$ is hamiltonian, for every $n \geq 3$.

\section{Notation and Terminology}\label{Section:NyT}

Let $G$ be a multigraph. If $e$ is an edge and $u$ and $v$ are the vertices such that $e=uv$, then $e$ is said to \textit{join} $u$ and $v$, we will say that $u$ and $v$ are the ends of $e$, we will say that the edge $e$ is incident with $u$ (respectively $v$) and also we will say that $u$ and $v$ are adjacent. If $u=v$, then the edge $e$ is a loop. Let $e$ and $f$ be two edges with the same end vertices, we will say $e$ and $f$ are parallel. Its said that $G$ is a \textit{looped} graph whenever for every vertex $v$ in $G$, there exists a loop $e$ incident with $v$. A graph $G'$ (multigraph or simple graph) is a subgraph (submultigraph or simple subgraph, respectively) of $G$ if $V(G') \subseteq V(G)$ and $E(G')\subseteq E(G)$. Let $S$ be a nonempty subset of $V(G)$; the subgraph (submultigraph or simple subgraph) of $G$ whose vertex set is $S$, and whose edge set is the set of those edges of G that have both ends in $S$, is called the subgraph (submultigraph or simple subgraph) of $G$ induced by $S$, and is denoted by $G[S]$.

A walk in a multigraph $G$ is a sequence $(v_0, e_0, v_1, e_1, \ldots , e_{k-1}, v_k)$, where $e_i = v_i v_{i+1}$ for every $i$ in $\{0, \ldots , k-1\}$. We will say that a walk is closed if $v_0 = v_k$. If $v_i \neq v_j$ for all $i$ and $j$ with $i \neq j$, it is called a path. A cycle is a closed walk $(v_0, e_0, v_1, e_1, \ldots , e_{k-1}, v_k , e_k, v_0)$, with $k \geq 3$, such that $v_i \neq v_j$ for all $i$ and $j$ with $i \neq j$. In a multigraph $G$ a walk in which no edge is repeated is a trail. A closed Euler trail in a multigraph $G$, is a closed walk which traverses each edge of $G$ exactly once. If $W=(v_0, e_0, v_1, e_1, \ldots , e_{k-1}, v_k)$ and $W' = (v_k , e_k, v_{k+1}, e_{k+1}, \ldots , e_{t-1}, v_t)$ are two walks, the walk $(v_k , e_{k-1}, \ldots , e_1, v_1, e_0, v_0)$, obtained by reversing $W$, is denoted by $W^{-1}$ and the walk $(v_0, e_0, v_1, \ldots , e_{k-1}, v_k,$  $e_k, v_{k+1}, e_{k+1}, \ldots , e_{t-1}, v_t)$, obtaining by concatenating $W$ and $W'$ at $v_k$, is denoted by $W \cup W'$.

A simple graph $G$ is said to be multipartite, if for some positive integer $k$, there exists
a partition ${X_1, \ldots , X_k }$ of $V(G)$, such that $X_i$ is an independent set in $G$ (that is no two vertices of $X_i$ are adjacent) for every $i$ in $\{1, \ldots , k\}$, in this case, also $G$ is called $k$-partite. It said that $G$ is a complete $k$-partite graph whenever $G$ is $k$-partite and for every $u$ in $X_i$ and for every $v$ in $X_j$ , with $i \neq j$ , we have that $u$ and $v$ are adjacent, denoted by $K_{n_1,\ldots,n_k}$ where $\vert X_i \vert = n_i$ for every $i$ in $\{1, \ldots, k\}$. In the particular case when $k=2$, the graph $G$ is said to be bipartite graph.
A matching in a graph $G$ is a subset $M$ of $E(G)$, such that no two elements of $M$ are incident
with the same vertex in $G$. A matching $M$ saturates a vertex $v$ if some edge of $M$ is
incident with $v$. If every vertex of $G$ is saturated, the matching $M$ is said to be perfect.
We will need the following results.

\begin{theorem}[Tutte \cite{tutte}]\label{theo:Tutte}
$G$ has a perfect matching if and only if $o(G\setminus S) \leq \vert S\vert$ for all proper subset $S$ of $V(G)$.
\end{theorem}

\begin{lemma}\label{lemma:Union-Perfect}
Let $M$ and $M'$ be two perfect matching of a graph $G$. If $M \cap M' = \emptyset$, then there exists a partition of the vertices of $G$ into even cycles. Moreover, every cycle alternate edges between $M$ and $M'$.
\end{lemma} 

\begin{figure}[tpb]
\centering 
	\subfigure[$G$ is an $H$-colored multigraph]
	{\scalebox{1}{\begin{tikzpicture}
	\tikzset{every loop/.style={min distance=15mm,in=120,out=60,looseness=10}}
	\tikzstyle{vertex}=[circle, draw=black] 
%%%%%%%%Gráfica G		
		\node[](G) at (0,2.5){\huge{$G$}};		
		\node[vertex](v1) at (1.5,2.5) {};
		\node[left=1mm of v1](Et.u) {$v_1$};%etiqueta de v1
		\node[vertex](v2) at (4,2.5) {};
		\node[right=1mm of v2](Et.u) {$v_2$};%etiqueta de v2
		\node[vertex](v3) at (1.5,0) {};
		\node[left=1mm of v3](Et.u) {$v_3$};%etiqueta de v3
		\node[vertex](v4) at (4,0) {};
		\node[right=1mm of v4](Et.u) {$v_4$};%etiqueta de v4
		
		\path	(v1) 	edge [above,blue,very thick] node {\textcolor{black}									{$e_1$}} (v2)
						edge [right,red,very thick,dashed] node {\textcolor{black}{$e_3$}} (v3)
						edge [bend right,left,blue,very thick] node
						{\textcolor{black}	{$e_2$}} (v3)
				(v2)	edge [left,blue,very thick] node {\textcolor{black}{$e_6$}} 						(v4)
						edge [bend left, right,red,very thick,dashed] node {\textcolor{black}		
				{$e_5$}} (v4)
				(v3)	edge [below,red,very thick,dashed] node {\textcolor{black}{$e_4$}} (v4);

%%%%%%%%Gráfica H
		\node[](H) at (0,-1){\huge{$H$}};
		\node[vertex,fill=blue](G) at (1.5,-2.5) {};
		\node[below=1mm of G]{$G$};		
		\node[vertex,fill=red](R) at (3.5,-2.5) {};
		\node[below=1mm of R] {$R$};
		
		\path	(G) edge [loop above] node {} ()
				(R) edge [loop above,dashed,thick] node {} ();		
		
	\end{tikzpicture}}}
	\hfill
%%%%%%%%%%%%Gráfica Ggv	
	\subfigure[$L_3^H(G)$]{\scalebox{1}{\begin{tikzpicture}
	\tikzset{every loop/.style={min distance=15mm,in=120,out=60,looseness=10}}
	\tikzstyle{vertex}=[circle, draw=black] 
		
		\node[vertex] (Gv1e1) at (-2, 2.5) {};
		\node[left=0mm of Gv1e1] {\small{$f(v_1,e_1)$}};
		\node[vertex] (Gv1e2) at (-4, 1) {};
		\node[left=0mm of Gv1e2] {\small{$f(v_1,e_2)$}};
		\node[vertex] (Gv1e3) at (-2, 1) {};
		\node[right=0mm of Gv1e3] {\small{$f(v_1,e_3)$}};
		\draw[dotted, thin] (-1.5,0.5) rectangle (-4.5,3);
		\node[](G1) at (-4.2,2.5){};	
		\node[left=1mm of G1] {\Large{$G_{v_1}$}};
		
		\node[vertex] (Gv2e1) at (1, 2.5) {};
		\node[right=0mm of Gv2e1] {\small{$f(v_2,e_1)$}};
		\node[vertex] (Gv2e6) at (1, 1) {};
		\node[vertex] (Gv2e5) at (3, 1) {};
		\draw[dotted, thin] (0.5,0.5) rectangle (3.5,3);		
		\node[](G2) at (3.2,2.5){};
		\node[right=1mm of G2] {\Large{$G_{v_2}$}};

		\node[vertex] (Gv3e3) at (-2, -1.5) {};
		\node[right=0mm of Gv3e3] {\small{$f(v_3,e_3)$}};
		\node[vertex] (Gv3e4) at (-2, -3) {};
		\node[vertex] (Gv3e2) at (-4, -1.5) {};
		\node[left=0mm of Gv3e2] {\small{$f(v_3,e_2)$}};
		\draw[dotted, thin] (-1.5,-1) rectangle (-4.5,-3.5);		
		\node[](G3) at (-4.2,-3){};
		\node[left=1mm of G3] {\Large{$G_{v_3}$}};
		
		\node[vertex] (Gv4e4) at (1, -3) {};
		\node[vertex] (Gv4e5) at (3, -1.5) {};
		\node[vertex] (Gv4e6) at (1, -1.5) {};
		\draw[dotted,thin] (0.5,-1) rectangle (3.5,-3.5);
		\node[](G4) at (3.2,-3){};
		\node[right=1mm of G4] {\Large{$G_{v_4}$}};
		
		\node[vertex] (me1) at (-0.5, 2.5) {};
		\node[below=0mm of me1] {\small{$m(e_1)$}};
		\node[vertex] (me2) at (-4, -0.25) {};
		\node[left=0mm of me2] {\small{$m(e_2)$}};
		\node[vertex] (me3) at (-2, -0.25) {};
		\node[right=0mm of me3] {\small{$m(e_3)$}};
		\node[vertex] (me4) at (-0.5, -3) {};
		\node[vertex] (me5) at (3, -0.25) {};
		\node[vertex] (me6) at (1, -0.25) {};
				
		\draw[] (Gv1e1) -- (me1);
		\draw[] (me1) -- (Gv2e1);
		\draw[] (Gv1e2) -- (me2);
		\draw[] (me2) -- (Gv3e2);
		\draw[] (Gv1e3) -- (me3);
		\draw[] (me3) -- (Gv3e3);
		\draw[] (Gv3e4) -- (me4);
		\draw[] (me4) -- (Gv4e4);
		\draw[] (Gv2e5) -- (me5);
		\draw[] (me5) -- (Gv4e5);
		\draw[] (Gv2e6) -- (me6);
		\draw[] (me6) -- (Gv4e6);

		\draw[] (Gv1e3) -- (Gv3e2);
		\draw[] (Gv1e2) -- (Gv3e3);
		\draw[] (Gv2e6) -- (Gv4e5);
		\draw[] (Gv2e5) -- (Gv4e6);

		\draw[] (Gv1e1) -- (Gv1e2);
		\draw[] (Gv2e1) -- (Gv2e6);
		\draw[] (Gv3e3) -- (Gv3e4);
		\draw[] (Gv4e4) -- (Gv4e5);
	\end{tikzpicture}}}
	\caption{The sequence $P=(v_1,e_1,v_2,e_6,e_5,v_4,e_4,v_3,e_3,e_2,v_1)$ is a closed Euler dynamic $H$-trail in $G$ but there is no closed Euler $H$-trail in $G$}
\label{fig:GV(G)}
\end{figure}
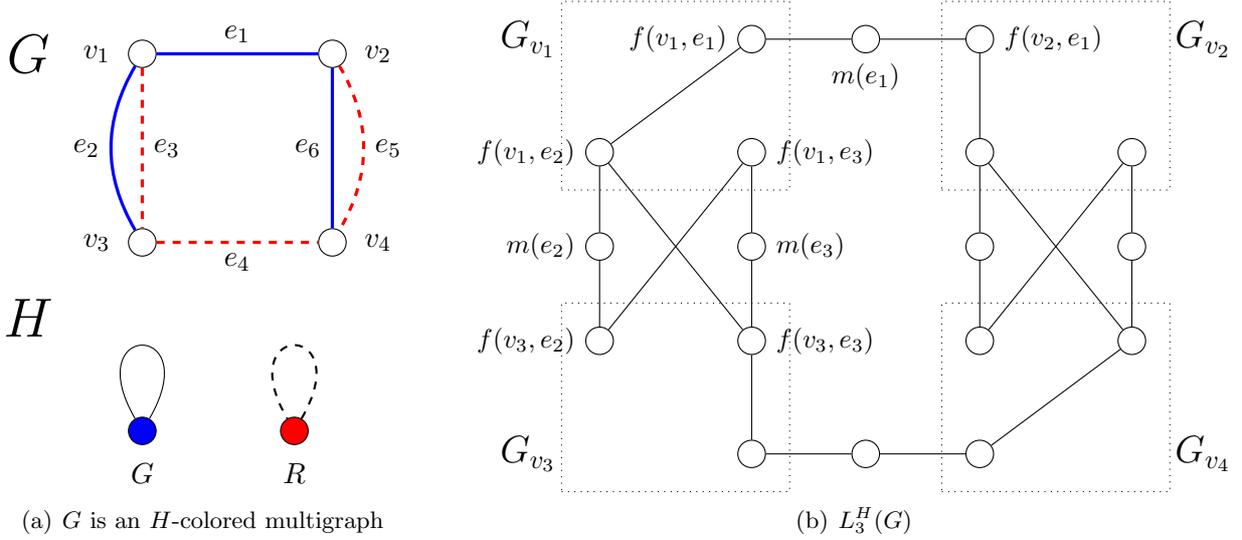

\section{Main results}
\label{Section:Main}

In what follows, $H$ will be a graph possibly with loops, and $G$ will be a multigraph without loops.

First of all, we will introduce an auxiliary graph, denoted by $L_n^H(G)$, which is defined as follows.

\begin{definition}
Let $G$ be an $H$-colored multigraph with $\vert E(G)\vert =q$. For $n \geq 2$, $L_n^H(G)$ is the graph with $nq$ vertices, obtained as follows: for each edge $e=uv$ of $G$, we take two vertices $f(u,e)$ and $f(v,e)$ in $L_n^H(G)$, and adding a path with $n-2$ new intermediate vertices connecting $f(u,e)$ and $f(v,e)$. And the rest of the edges of $L_n^H(G)$ are defined as follows: a) $f(u,e)$ and $f(u,g)$ are adjacent iff $e \neq g$ and $c(e)c(g) \in E(H)$; b) $f(u,e)$ and $f(v,g)$ are adjacent iff $u \neq v$ and $e$ and $g$ are parallel in $G$.
\end{definition}

Notice that $L_n^H(G)$ is a simple graph.

Before giving a new way to construct the graph $L_n^H(G)$, we first introduce some additional notation.

Recall that $G_u$ is the graph such that $V(G_u)=\{e \in E(G) : e \textrm{ is incident}$ $\textrm{with } u\}$, and two different vertices $a$ and $b$ are joining by only one edge in $G_u$ if and only if $c(a)$ and $c(b)$ are adjacent in $H$.

For each $e=xy$ in $E(G)$, by the definition of $G_x$ and $G_y$, we have that $e \in V(G_x)$ and $e \in V(G_y)$. So, we will say that $f(x,e)$ and $f(y,e)$ are the copies of $e$ seen as a vertex in $G_x$ and $G_y$, respectively; and if $E_{xy}$ is the set of all the edges with end vertices $x$ and $y$, then we will say that $E_{xy}^x=\{f(x,e) : e \in E_{xy}\}$. 

Notice that $L_n^H(G)$ can be constructed as follows. First take the disjoint union of $G_x$, for every $x$ in $V(G)$. Then, for every pair of distinct vertices, $x$ and $y$ in $V(G)$, such that $\vert E_{xy}\vert =m \geq 1$, we will have a complete bipartite graph $K_{m,m}$ between $E_{xy}^x$ and $E_{xy}^y$. Finally, for every $e=xy \in E(G)$, we change the edge joining $f(x,e)$ and $f(y,e)$ by the path with $n-2$ new intermediate vertices. The construction of $L_3^H(G)$ is illustrated in Figure \ref{fig:GV(G)}.

It follows from the definition of $L_2^H(G)$ that $M_J=\{f(x,e)f(y,e) \in E(L_2^H(G)) : e=xy \in E(G)\}$ is a perfect matching, that we will called it \textit{joint matching of} $L_2^H(G)$.

The following two lemmas show how to construct a cycle in $L_2^H(G)$, based on the order of the edges in a closed dynamic H-trail in $G$, and vice versa.

\begin{lemma}\label{lemma:H-trail to cycle}
Let $G$ be an $H$-colored multigraph. If $P=(x_0,e_1,\ldots, e_{p_0},x_1, e_{p_0+1},$ $\ldots,e_{p_0+\ldots+p_1},x_2,$ $\ldots, x_{n-1},$ $e_{p_0+\ldots+p_{n-2}+1},\ldots,$ $e_{p_0+\ldots+p_{n-1}},x_n)$ is a closed dynamic $H$-trail in $G$, then $C=(f(x_0,e_1),$ $f(x_1,e_1),\ldots,$ $f(x_0,e_{p_0}), f(x_1,e_{p_0}), f(x_1,e_{p_0+1}),f(x_2,e_{p_0+1})\ldots,f(x_n,e_{p_0+\ldots+p_{n-1}}) ,f(x_0,e_1) )$ is a \newline cycle in $L_2^H(G)$.  
\end{lemma}
\begin{proof}
It follows from the definition of $L_2^H(G)$ that $f(e,x_i)f(e,x_{i+1})$ is an edge of $L_2^H(G)$, for every  $e=x_ix_{i+1} \in E(P)$.

Let $e_i$ and $e_{i+1}$ be consecutive edges in $P$ (if $i=p_0+\ldots+p_{n-1}$, then $e_{i+1}=e_1$). If $e_i$ and $e_{i+1}$ are parallel, such that $x_j$ and $x_{j+1}$ are the ends of both edges, then $f(x_{j+1},e_i)$ and $f(x_j,e_{i+1})$ are adjacent in $L_2^H(G)$, by the definition of $L_2^H(G)$. Otherwise, $e_i$ and $e_{i+1}$ are incident with $x_j$, for some $x_j$ in $V(P)$. Since $P$ is a closed dynamic $H$-trail, then $c(e_i)c(e_{i+1})$ is an edge in $H$ and $f(x_j,e_i)f(x_j,e_{i+1}) \in E(L_2^H(G))$. Hence, $C$ is a walk in $L_2^H(G)$. Since, $P$ is a closed dynamic $H$-trail, then $C$ does not repeat vertices and, so $C$ is a cycle. Moreover, $C$  alternate edges between $E(L_2^H(G))\setminus M_J$ and $M_J$.
\end{proof}

\begin{lemma}\label{lemma:Cycle to H-trail}
Let $G$ be an $H$-colored multigraph and $M_J$ be the joint matching of $L_2^H(G)$. If $C=(f(x_1,e_1),$ $f(y_1,e_1),f(x_2,e_2),f(y_2,e_2)\ldots,f(y_q,e_q), f(x_1,e_1))$, where $e_i=x_iy_i \in E(G)$, be a cycle in $L_2^H(G)$ that alternate edges between $E(L_2^H(G))\setminus M_J$ and $M_J$, then the following steps produces a closed dynamic $H$-trail in $G$.
\begin{enumerate}
\item Star with $P=(x_1,e_1,y_1)$ and $k=2$.

\item Let $e_k$, with end vertices $x_k$ and $y_k$. If $e_{k-1}$ and $e_k$ are parallel, $x_k=x_{k-1}$ and $y_k=y_{k-1}$, then we change lanes from the edge $e_{k-1}$ to $e_k$ in $P$. Otherwise, continue $P$ by the edge $e_k$.

\item $k=k+1$.

\item If $k=q+1$, finish. Otherwise, go to step 2.
\end{enumerate}

\end{lemma}
\begin{proof}

First, we will prove that $P$ is a dynamic $H$-trail in $G$.

Let $e_{k-1}$ and $e_k$, with $2 \leq k \leq q$.

Since $f(y_{k-1},e_{k-1})f(x_k,e_k) \in E(C)$, we have that $f(y_{k-1},e_{k-1})f(x_k,e_k)$ is in $E(L_2^H(G))$. Hence, $y_{k-1}=x_k$ and $c(e_{k-1})c(e_k)\in E(H)$ or $y_{k-1} \neq x_k$ and $e_{k-1}$ and $e_k$ are parallel in $G$.

If $y_{k-1}=x_k$ and $c(e_{k-1})c(e_k) \in E(H)$, then $e_{k-1}$ and $e_k$ are incident in $y_{k-1}=x_k$ and $(x_1,e_1,P,x_{k-1},$ $e_{k-1},y_{k-1}=x_k,e_k,y_k)$ is a dynamic $H$-walk in $G$.

Otherwise, $y_{k-1} \neq x_k$ and $e_{k-1}$ and $e_k$ are parallel edges in $G$, and $(x_1,e_1,P,$ $x_{k-1}=x_k,e_{k-1},e_k,$ $y_{k-1}=y_k)$ is a dynamic $H$-walk in $G$.

Hence, $P$ is a dynamic $H$-walk in $G$.

On the other hand, since $C$ is a cycle, then $e_i$ appears exactly once in $P$, for every $i \in \{1,\ldots,q\}$, and so $P$ is a dynamic $H$-trail.

Now, we prove that $P$ is closed.

Recall that $e_i=x_iy_i$, for every $i\in \{1,\ldots,q\}$, and $y_i=x_{i+1}$ or $x_i=x_{i+1}$ and $y_i=y_{i+1}$ (if $i=q+1$, then $x_{i+1}=x_1$ and $y_{i+1}=y_1$). We will consider two cases.

\textbf{Case 1.} $x_i=x_j$ and $y_i=y_j$, for every pair of distinct elements, $i$ and $j$, in $\{1,\ldots,q\}$.

It follows from the construction of $P$ that $P=(x_1,e_1,\ldots,e_q,y_1)$. Since $C$ is a cycle in $L_2^H(G)$, then $q \geq 2$ and, so $P$ is closed.

\textbf{Case 2.} There exists $i$ in $\{1,\ldots,q\}$ such that $y_i = x_{i+1}$.

If $y_q=x_1$, then $c(e_q)c(e_1) \in E(H)$. Therefore, $P$ is closed.

Otherwise, $e_q$ and $e_1$ are parallel and, so $P$ is closed.\vspace{1em}

Therefore, $P$ is a closed dynamic $H$-trail in $G$. 
\end{proof}

It follows from Lemmas \ref{lemma:H-trail to cycle} and \ref{lemma:Cycle to H-trail} the following result.

\begin{theorem}\label{theo:Bijection}
Let $G$ be an $H$-colored multigraph and $M_J$ be the joint matching. Then, there is a bijection bewteen the set of closed dynamic $H$-trails in $G$ and the set of cycles in $L_2^H(G)$ that alternate edges between $E(L_2^H(G))\setminus M_J$ and $M_J$.
\end{theorem}

\begin{proof}
Let $G$ be an $H$-colored multigraph and $M_J$ the joint matching of $L_2^H(G)$. We will denote by $\mathcal{P}$ the set of closed dynamic $H$-trails in $G$ and by $\mathcal{C}$ the set of cycles in $L_2^H(G)$ that alternate edges between $E(L_2^H(G))\setminus M_J$ and $M_J$.

Let $T: \mathcal{P} \rightarrow \mathcal{C}$ defined by $T(P)=C$, where $P=(x_0,e_1,\ldots, e_{p_0},x_1,e_{p_0+1},\ldots,e_{p_1},x_2,e_{p_1+1},\ldots ,x_{n-1},$ $e_{p_0+\ldots+p_{n-2}+1},\ldots,$ $e_{p_0+\ldots+p_{n-1}},x_n)$, and $C=(f(x_0,e_1),f(x_1,e_1),\ldots,f(x_0,e_{p_0}), f(x_1,e_{p_0}),f(x_1,e_{p_0+1}),\ldots,$ $f(x_n,e_{p_0+\ldots+p_{n-1}}),f(x_0,e_1) )$.

\textbf{Claim 1.} $T$ is well-defined.

It follows from  Lemma \ref{lemma:H-trail to cycle} that $T(P) \in \mathcal{C}$.

Let $P_1=(x_0,e_0,\ldots, e_{p_0},x_1,e_{p_0+1},\ldots,e_{p_0+p_1},x_2,\ldots,x_{n-1},e_{p_0+\ldots+p_{n-2}+1},\ldots,e_{p_0+\ldots+p_{n-1}},x_n)$ and $P_2=(y_0,f_0,\ldots,f_{q_0},y_1,f_{q_0+1},\ldots, f_{q_0+q_1},y_2,\ldots,y_{n-1},f_{q_0+\ldots+q_{n-2}+1}, \ldots,f_{q_0+\ldots+q_{n-1}},y_n)$ in $\mathcal{P}$ such that $P_1 = P_2$.

Since $P_1=P_2$, then $E(P_1)=E(P_2)$ and there exists $k \in \mathbb{N}$ such that $e_i=f_{i+k (mod  \,p_0+\ldots+p_{n-1})}$ (since the edges are traversed in the same order but the first edge is not the same). It follows from the definition of $T$ that $V(T(P_1))=V(T(P_2))$ and the order of the vertices are the same. Then, we have that $T(P_1)=T(P_2)$. Therefore, $T$ is well-defined.

\textbf{Claim 2.} $T$ is injective.

Let $P_1$ and $P_2$ in $\mathcal{P}$ such that $P_1 \neq P_2$. 

If $E(P_1) \neq E(P_2)$, then $V(T(P_1)) \neq V(T(P_2))$ and, so $T(P_1) \neq T(P_2)$. Otherwise, since $P_1 \neq P_2$, then there exists $\{e_i,e_{i+1}\} \subseteq E(P_1) = E(P_2)$ such that $e_i$ is the edge preceding $e_{i+1}$ in $P_1$ and $e_i$ is not the edge preceding $e_{i+1}$ in $P_2$. Therefore, $T(P_1) \neq T(P_2)$ and $T$ is injective.

\textbf{Claim 3.} $T$ is surjective.

Let $C$ be a cycle in $\mathcal{C}$. Since $C$ alternate edges between $E(L_2^H(G))\setminus M_J$ and $M_J$, $C$ must be of the form $C=(f(x_1,e_1),f(y_1,e_1),\ldots,f(x_q,e_q),$ $f(y_q,e_q), f(x_1,e_1))$, where $e_i=x_iy_i \in E(G)$.

It follows from the definition of $T$ that $T(P)=C$, where $P$ be the closed dynamic $H$-trail obtained by applying the Lemma \ref{lemma:Cycle to H-trail} to the cycle $C$. Hence, $T$ is surjective.

Therefore, $T$ is a bijection.
\end{proof}

\begin{theorem}\label{theo:2}
Let $G$ be an $H$-colored multigraph and $M_J$ be the joint matching of $L_2^H(G)$. There exists a perfect matching $M$ in $L_2^H(G)\setminus M_J$ if and only if there exists a partition of the edges of $G$ into closed dynamic $H$-trails. (Notice that some of the closed dynamic $H$-trails can be of the form $(x,e_0,\ldots,e_k,y)$, where $k\geq 1$). 
\end{theorem}
\begin{proof}
Let $G$ be an $H$-colored multigraph and $M_J$ be the joint matching of $L_2^H(G)$.

Suppose that there exists a perfect matching $M$ in $L_2^H(G)\setminus M_J$. 

It follows by Lemma \ref{lemma:Union-Perfect}, that there exists a partition of the vertices of $L_2^H(G)$ into vertex-disjoint cycles, say $\mathfrak{C}=\{C_1,\ldots,C_n\}$. Moreover, each cycle in $\mathfrak{C}$ alternate edges between $M_J$ and $M$.

Let $\mathfrak{P}=\{P_1,\ldots, P_n\}$ be the set of closed dynamic $H$-trails in $G$ such that $T(P_i)=C_i$, for every $i \in \{1,\ldots,n\}$.\vspace{1em}

\textbf{Claim 1} $E(P_i) \cap E(P_j) = \emptyset$, for every $\{i,j\} \subseteq \{1,\ldots, n\}$.

Proceeding by contradiction, suppose that there exists a pair of distinct elements, $i$ and $j$ in $\{1,\ldots,n\}$, such that $E(P_i) \cap E(P_j) \neq \emptyset$.

Let $e=xy \in E(P_i) \cap E(P_j)$, it follows from the definition of $T$ that $f(x,e)$ and $f(y,e)$ are vertices in $V(C_i)\cap V(C_j)$, a contradiction.\vspace{1em}

\textbf{Claim 2} $\bigcup_{i=1}^n E(P_i)=E(G)$.

By the definition of $T$, we have that $\bigcup_{i=1}^n \; E(P_i) \subseteq E(G)$.  

On the other hand, let $e \in E(G)$ join $x$ and $y$. Since, $\mathfrak{C}$ is a partition of the vertices of $L_2^H(G)$, then there exists a cycle in $\mathfrak{C}$, say $C_j$, such that $f(x,e)f(y,e) \in E(C_j)$. It follows from the construction of $P_j$ that $e \in E(P_j)$. Therefore, $e \in \bigcup_{i=1}^n E(P_i)$ and $E(G) \subseteq \bigcup_{i=1}^n E(P_i)$.\vspace{1em}

Therefore, $\mathfrak{P}=\{P_1,\ldots, P_n\}$ is a partition of the edges of $G$ into closed dynamic $H$-trails.\vspace{1em}

Conversely, suppose that there exists a partition of the edges of $G$ into closed dynamic $H$-trails, say $\mathfrak{P}=\{P_1,\ldots, P_k\}$.

Let $\mathfrak{C}=\{C_1, \ldots, C_k\}$ a set of cycles in $L_2^H(G)$ such that $T(P_i)=C_i$, for every $i \in \{1,\ldots, k\}$.

\textbf{Claim 3} If $i \neq j$, then $V(C_i) \cap V(C_j) = \emptyset$.

Suppose, to the contrary, that there exists $i$ and $j$ in $\{1,\ldots,k\}$, such that $V(C_i) \cap V(C_j) \neq \emptyset$. Let $f(x,e) \in V(C_i) \cap V(C_j)$, for some $e \in E(G)$ and $x \in V(G)$, such that $e$ is incident with $x$.

It follows from the construction of $C_i$ and $C_j$ that $e \in E(P_i)$ and $e \in E(P_j)$, a contradiction.

\textbf{Claim 4} $\bigcup_{i=1}^k V(C_i)=V(L_2^H(G))$.

By the construction of $C_i$, we have that $\bigcup_{i=1}^k V(C_i) \subseteq V(L_2^H(G))$.

Let $f(x,e) \in V(L_2^H(G))$, with $e \in E(G)$ and $x \in V(G)$, such that $e$ is incident with $x$. 

Since $\mathfrak{P}$ is a partition of the edges of $G$, then there exists $P_i \in \mathfrak{P}$ such that $e \in E(P_i)$. It follows from the construction of $C_i$ that $f(x,e) \in V(C_i) \subseteq \bigcup_{j=1}^k V(C_j)$. Hence, $V(L_2^H(G)) \subseteq \bigcup_{j=1}^k V(C_j)$. 

Therefore, $\bigcup_{i=1}^k V(C_i)=V(L_2^H(G))$.\vspace{1em}

\textbf{Claim 5} $M= \bigcup_{i=1}^k E(C_i) \setminus M_J$ is a perfect matching in $L_2^H(G)$. 

Recall $M_J=\{f(x,e)f(y,e) \in E(L_2^H(G)) : e=xy \in E(G)\}$.

It follows from the construction of $C_i$ and claims 3 and 4 that $M$ is a perfect matching in $L_2^H(G)$.\vspace{1em}

Therefore, $M$ is a perfect matching in $L_2^H(G)$, such that $M \cap M_J = \emptyset$.
\end{proof}

\begin{theorem}\label{theo:3}
Let $G$ be an $H$-colored multigraph. Then:
\renewcommand{\labelenumi}{\alph{enumi}.}
\begin{enumerate}
\item If $G$ has a closed Euler dynamic $H$-trail, then $L_n^H(G)$ is hamiltonian, for every $n \geq 2$.
\item If $L_n^H(G)$ is hamiltonian, for some $n\geq 3$, then $G$ has a closed Euler dynamic $H$-trail.
\item $G$ has a closed Euler dynamic $H$-trail if and only if $L_n^H(G)$ is hamiltonian, for every $n \geq 3$.
\end{enumerate}
\end{theorem}
\begin{proof}
Let $G$ be an $H$-colored multigraph.

a. It follows from Theorems \ref{theo:Bijection} and \ref{theo:2} and definition of $L_n^H(G)$.

b. Suppose that there exist $n \geq 3$, such that $L_n^H(G)$is hamiltonian. Let $C'$ be a  hamiltonian cycle in $L_n^H(G)$.

Then, every path from $f(x,e)$ to $f(y,e)$ is in $V(C')$ in some order, for every $e=xy$ in $E(G)$, by the definition of $L_n^H(G)$. Hence, we replace the path for the edge $f(x,e)f(y,e)$ to $C'$, and we get a Hamilton cycle in $L_2^H(G)$, say $C$.
 
Since $C$ is a Hamilton cycle, then $P=T^{-1}(C)$ is a closed Euler dynamic $H$-trail in $G$, by Theorem \ref{theo:2}.

c. It follows from a and b.
\end{proof}

In Theorem \ref{theo:3}, we cannot change $n \geq 3$ by $n \geq 2$ because there are $H$-colored multigraphs without closed Euler dynamic $H$-trail and $L_2^H(G)$ is hamiltonian, see Figure \ref{fig:Conter}.

\begin{figure}[tpb]
\centering 
	\subfigure[$G$ is an $H$-colored multigraph]
	{\scalebox{1}{\begin{tikzpicture}
	\tikzset{every loop/.style={min distance=15mm,in=120,out=60,looseness=10}}
	\tikzstyle{vertex}=[circle, draw=black] 
%%%%%%%%Gráfica G		
		\node[](G) at (0,2.5){\huge{$G$}};		
		\node[vertex](v1) at (1,2.5) {};
		\node[above=0mm of v1](Et.u) {$v_1$};%etiqueta de v1
		\node[vertex](v2) at (3,2.5) {};
		\node[above=0mm of v2](Et.u) {$v_2$};%etiqueta de v2
		\node[vertex](v3) at (5,2.5) {};
		\node[above=0mm of v3](Et.u) {$v_3$};%etiqueta de v3
		\node[vertex](v4) at (1,0) {};
		\node[below=0mm of v4](Et.u) {$v_4$};%etiqueta de v4
		\node[vertex](v5) at (3,0) {};
		\node[below=0mm of v5](Et.u) {$v_5$};%etiqueta de v5
		\node[vertex](v6) at (5,0) {};
		\node[below=0mm of v6](Et.u) {$v_6$};%etiqueta de v6
		
		\draw[red,dashed] (v1) -- node[below] {\textcolor{black}{$e_1$}} ++(v2);
		\draw[red,dashed] (v1) -- node[right] {\textcolor{black}{$e_2$}} ++(v4);
		\draw[blue] (v2) -- node[below] {\textcolor{black}{$e_6$}} ++(v3);
		\draw[blue] (v3) -- node[left] {\textcolor{black}{$e_7$}} ++(v6);
		\draw[red,dashed] (v4) -- node[above] {\textcolor{black}{$e_3$}} ++(v5);
		\draw[blue] (v5) -- node[above] {\textcolor{black}{$e_8$}} ++(v6);
		\path (v2) 	edge [right,bend left,blue] node {\textcolor{black}												{$e_5$}} (v5)
					edge [left,bend right,red,,dashed] node {\textcolor{black}												{$e_4$}} (v5);
		
%%%%%%%%Gráfica H
		\node[](H) at (0,-1){\huge{$H$}};
		\node[vertex,fill=blue](G) at (1.5,-2.5) {};
		\node[below=1mm of G]{$G$};		
		\node[vertex,fill=red](R) at (3.5,-2.5) {};
		\node[below=1mm of R] {$R$};
		
		\path	(G) edge [loop above] node {} ()
				(R) edge [loop above,dashed,thick] node {} ();		
		
	\end{tikzpicture}}}
	\hfill
%%%%%%%%%%%%Gráfica Ggv	
	\subfigure[$L_2^H(G)$]{\scalebox{1}{\begin{tikzpicture}
	\tikzset{every loop/.style={min distance=15mm,in=120,out=60,looseness=10}}
	\tikzstyle{vertex}=[circle, draw=black] 
		
		\node[](L2) at (1,6){\Large{$L_2^H(G)$}};
		\node[vertex](e1v1) at (2,5) {};
%		\node[above=0mm of e1v1](Et.u) {$v_1$};%etiqueta de v1
		\node[vertex](e1v2) at (3.5,5) {};
%		\node[above=0mm of e1v2](Et.u) {$v_2$};%etiqueta de v2
		\node[vertex](e2v1) at (1,4) {};
%		\node[above=0mm of e2v1](Et.u) {$v_3$};%etiqueta de v3
		\node[vertex](e2v4) at (1,2) {};
%		\node[below=0mm of e2v4](Et.u) {$v_4$};%etiqueta de v4
		\node[vertex](e3v4) at (2,1) {};
%		\node[below=0mm of e3v4](Et.u) {$v_3$};%etiqueta de v5
		\node[vertex](e3v5) at (3.5,1) {};
%		\node[below=0mm of e3v5](Et.u) {$v_4$};%etiqueta de v6
		\node[vertex](e4v2) at (4.5,4) {};
%		\node[above=0mm of e4v2](Et.u) {$v_1$};%etiqueta de v1
		\node[vertex](e4v5) at (4.5,2) {};
%		\node[above=0mm of e4v5](Et.u) {$v_2$};%etiqueta de v2
		\node[vertex](e5v2) at (6,4) {};
%		\node[above=0mm of e5v2](Et.u) {$v_3$};%etiqueta de v3
		\node[vertex](e5v5) at (6,2) {};
%		\node[below=0mm of e5v5](Et.u) {$v_4$};%etiqueta de v4
		\node[vertex](e6v2) at (7,5) {};
%		\node[below=0mm of e6v2](Et.u) {$v_3$};%etiqueta de v5
		\node[vertex](e6v3) at (8.5,5) {};
%		\node[below=0mm of e6v3](Et.u) {$v_4$};%etiqueta de v6
		\node[vertex](e7v3) at (9.5,4) {};
%		\node[below=0mm of e7v3](Et.u) {$v_3$};%etiqueta de v5
		\node[vertex](e7v6) at (9.5,2) {};
%		\node[below=0mm of e7v6](Et.u) {$v_4$};%etiqueta de v6
		\node[vertex](e8v5) at (7,1) {};
%		\node[below=0mm of e8v5](Et.u) {$v_3$};%etiqueta de v5
		\node[vertex](e8v6) at (8.5,1) {};
%		\node[below=0mm of e8v6](Et.u) {$v_4$};%etiqueta de v6
		\node[](b) at (1,0.5) {};

		\draw[] (e1v1) -- (e1v2);
		\draw[] (e1v1) -- (e2v1);
		\draw[] (e2v1) -- (e2v4);
		\draw[] (e2v4) -- (e3v4);
		\draw[] (e3v4) -- (e3v5);
		\draw[] (e3v5) -- (e4v5);
		\draw[] (e4v5) -- (e4v2);
		\draw[] (e4v5) -- (e5v2);
		\draw[] (e4v2) -- (e5v5);
		\draw[] (e4v2) -- (e1v2);
		\draw[] (e5v2) -- (e5v5);
		\draw[] (e5v2) -- (e6v2);
		\draw[] (e6v2) -- (e6v3);
		\draw[] (e6v3) -- (e7v3);
		\draw[] (e7v3) -- (e7v6);
		\draw[] (e7v6) -- (e8v6);
		\draw[] (e8v6) -- (e8v5);
		\draw[] (e8v5) -- (e5v5);
		
	\end{tikzpicture}}}
	\caption{The graph $L_2^H(G)$ is hamiltonian but there is no closed Euler dynamic $H$-trail in $G$}
\label{fig:Conter}
\end{figure}
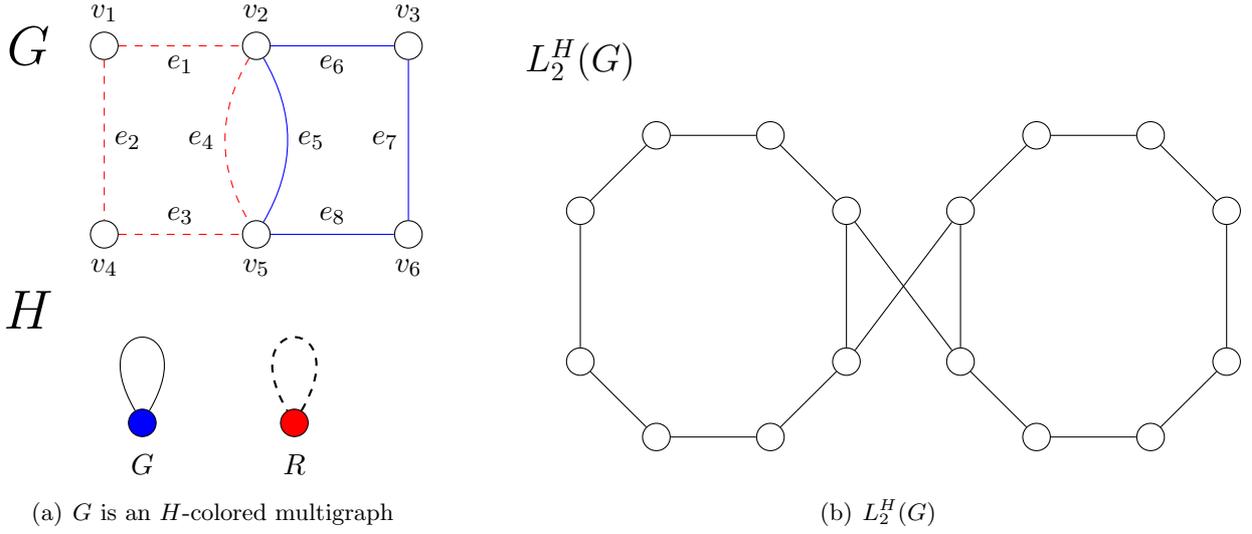

Recall that $E_{xy}$ is the set of all the edges with end vertices $x$ and $y$.

\begin{definition}
Let $G$ be an $H$-colored multigraph and $M$ be a matching in $L_2^H(G)$. For each $x$ in $V(G)$, we will say that $M_x=\{f(x,e)f(x,g) \in M\}$, i.e., $M_x$ consists of the edges in $G_x$ that are in $M$. And, for every pair of vertices $x$ and $y$ in $V(G)$, such that $E_{xy} \neq \emptyset$, we will define the following sets: $A_{xy}^M = \{e \in E_{xy} : f(x,e)$ is $M_x$-saturate and $f(y,e)$ is not $M_y$-saturate$\}$; $B_{xy}^M = \{e \in E_{xy} : f(y,e)$ is $M_y$-saturate and $f(x,e)$ is not $M_x$-saturate$\}$; $C_{xy}^M = \{e \in E_{xy} : f(x,e)$ is $M_x$-saturate and $f(y,e)$ is $M_y$-saturate$\}$; and $D_{xy}^M = \{e \in E_{xy} : f(x,e)$ is not $M_x$-saturate and $f(y,e)$ is not $M_y$-saturate$\}.$
\end{definition}

\begin{theorem}\label{theo:4}
Let $G$ be an $H$-colored multigraph. There exists a partition of the edges of $G$ into closed dynamic $H$-trail, none of them in the form $(x,a_1,\ldots,a_l,y)$, $l \geq 2$, if and only if there exists a non empty matching, say $M$, in $L_2^H(G)$, such that for each set $E_{xy} \neq \emptyset$ one of the following conditions hold: a) $C_{xy}^M=E_{xy}$; b) $\vert C_{xy}^M \vert < \vert E_{xy}\vert$ and $1 \leq \vert A_{xy}^M \vert = \vert B_{xy}^M \vert$.
\end{theorem}
\begin{proof}
Let $G$ be an $H$-colored multigraph and $M_J$ be the joint matching of $L_2^H(G)$.

Suppose that $\mathfrak{P}=\{P_1,\ldots,P_k\}$ is a partition of the edges of $G$ into closed dynamic $H$-trails, such that none of them have the form $(x,a_1,\ldots,a_l,y)$, $l \geq 2$.

By Theorem \ref{theo:2}, there exists a perfect matching in $L_2^H(G)\setminus M_J$. Let $M$ be the perfect matching obtained as in the proof of Theorem \ref{theo:2}, i.e, $M= \bigcup_{i=1}^k E( T(P_i)) \setminus M_J$.

Let $x$ and $y$ be a pair of vertices in $G$, such that $E_{xy}=\{e_1,\ldots,e_q\}$, with $q \geq 1$. Since $\mathfrak{P}$ is a partition of the edges of $G$ into closed dynamic $H$-trails, then there exists $P_i \in \mathfrak{P}$ such that $E_{xy} \cap E(P_i) \neq \emptyset$, say $P_i=P_1$.

Since $P_1$ is not of the form $(x,a_1,\ldots,a_l,y)$, $l \geq 2$, and $E_{xy} \cap E(P_1) \neq \emptyset$, then $P_1=(\ldots, f_0,x,e_{t_1},\ldots,e_{t_{p_1}},$ $y,f_1,\ldots)$, where $p_1 \leq q$ and $\{f_0,f_1\} \subset E(G)$. (Recall that $M_x$ consists of the edges in $G_x$ that are in $M$).

If $p_1=1$, then $P_1=(\ldots,f_0,x,e_{t_1},y,f_1,\ldots)$. Hence, $f(x,f_0)f(x,e_{t_1})$ and $f(y,e_{t_1})f(y,f_1)$ are in $M$. 

Therefore, $f(x,f_0)f(x,e_{t_1}) \in M_x$, $f(y,e_{t_1})f(y,f_1) \in M_y$ and $e_{t_1} \in C_{xy}^M$.

On the other hand, if $p_1 \geq 2$, then $f(x,f_0)f(x,e_{t_1})$, $f(y,e_{t_{p_1}})f(y,f_1)$ and $f(y,e_{t_i})f(x,e_{t_{i+1}})$ are in $M$, for every $i \in \{1,\ldots,p_1-1\}$, see figure \ref{fig:case2}. Hence, $e_{t_1} \in A_{xy}^M$, $e_{t_{p_1}} \in B_{xy}^M$ and $e_{t_i} \in D_{xy}^M$, for every $i \in \{2,\ldots,p_1-1\}$.

\begin{figure}[tpb]
\centering 
\subfigure[$P_1=(\ldots,f_0,x,e_1,\ldots,e_{p_1},y,f_1,\ldots)$]
	{\scalebox{1}{\begin{tikzpicture}
	\tikzset{every loop/.style={min distance=15mm,in=120,out=60,looseness=10}}
	\tikzstyle{vertex}=[circle, draw=black] 
		
		\node[vertex] (x) at (1.5,0) {};
		\node[below=1mm of x] {$x$};
		\node[vertex] (y) at (4.5,0) {};
		\node[below=1mm of y] {$y$};
		\node[vertex] (x0) at (0,0) {};
		\node[vertex] (y0) at (6,0) {};
		\node[] (f2) at (3,0.4) {};
		\node[] (f3) at (3,-0.4) {};
		
		\path	(x)	edge [bend left=60,thick,above] node {$e_1$} (y)
					edge [bend left=20,thick,above] node {$e_2$} (y)
					edge [bend right=20,thick,below] node {$e_{p_1-1}$} (y)
					edge [bend right=60,thick,below] node {$e_{p_1}$} (y)
				(x0)edge [thick,below] node {$f_0$} (x)
				(y)edge [thick,below] node {$f_1$} (y0);
		\draw[dotted, very thick] (f2) -- (f3);
	\end{tikzpicture}}}
	\hfill
	\subfigure[Cycle in $L_2^H(G)$ obtained from $P_1$.\label{Subfig:Caso1}]
	{\scalebox{1}{\begin{tikzpicture}
	\tikzset{every loop/.style={min distance=15mm,in=120,out=60,looseness=10}}
	\tikzstyle{vertex}=[circle, draw=black] 
		
		\node[vertex] (ex1) at (0,0) {};
		\node[below=0mm of ex1] {$f(x,e_1)$};
		\node[vertex] (ey1) at (0,2) {};
		\node[above=0mm of ey1] {$f(y,e_1)$};
				
		\node[vertex] (ex2) at (1.5,0) {};
		\node[below=0mm of ex2] {$f(x,e_2)$};
		\node[vertex] (ey2) at (1.5,2) {};
		\node[above=0mm of ey2] {$f(y,e_2)$};

		\node[vertex] (exk) at (3.5,0) {};
		\node[below=0mm of exk] {$f(x,e_{p_1-1})$};
		\node[vertex] (eyk) at (3.5,2) {};
		\node[above=0mm of eyk] {$f(y,e_{p_1-1})$};

		\node[vertex] (fx1) at (5.2,0) {};
		\node[below=0mm of fx1] {$f(x,e_{p_1})$};
		\node[vertex] (fy1) at (5.2,2) {};
		\node[above=0mm of fy1] {$f(y,e_{p_1})$};

		\node[vertex] (fx0) at (-1.7,1) {};
		\node[below=1mm of fx0] {$f(x,f_0)$};
		\node[vertex] (fy0) at (6.5,1) {};
		\node[below=1mm of fy0] {$f(y,f_1)$};

		\node[] (g00) at (2,0) {};
		\node[] (g01) at (3,0) {};
		\node[] (g02) at (2.5,0) {};
		\node[] (g10) at (2,2) {};
		\node[] (g11) at (3,2) {};
		\node[] (g12) at (2.5,2) {};
				
		\path	(fx0)	edge [] node {} (ex1)
				(ey1)	edge [] node {} (ex2)
				(eyk)	edge [] node {} (fx1)
				(fy1)	edge [] node {} (fy0)
				(ey2)	edge [] node {} (g02)
				(exk)	edge [] node {} (g12);
		
		\path	(ex1)	edge [dashed] node {} (ey1)
				(ex2)	edge [dashed] node {} (ey2)
				(exk)	edge [dashed] node {} (eyk)
				(fx1)	edge [dashed] node {} (fy1);
		
		\draw[dotted] (g00) --(g01);
		\draw[dotted] (g10) --(g11);

	\end{tikzpicture}}}
	
	\caption{In b, the continued edges are in $M$ and the dashed edges are in $M_J$}
	\label{fig:case2}
\end{figure}
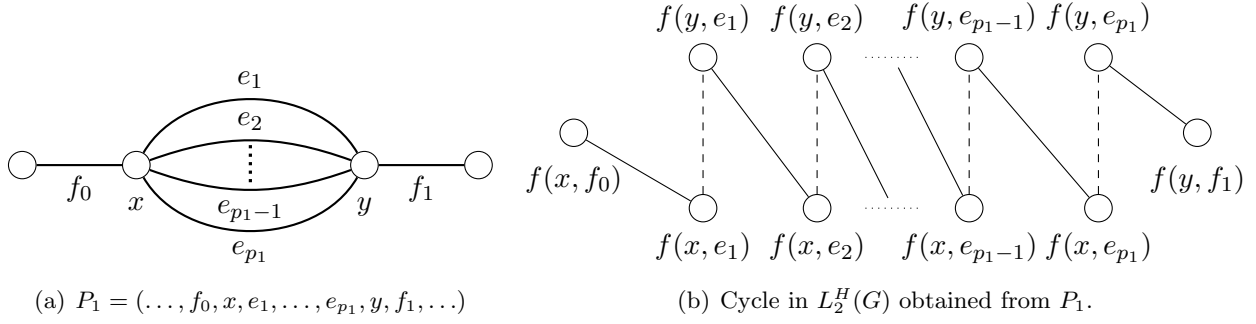

If $E_{xy}\setminus \{e_{t_1},\ldots,e_{t_{p_1}}\}= \emptyset$, then  $C_{xy}^M=E_{xy}$ (when $p_1=1$ occurs) or $\vert C_{xy}^M \vert < \vert E_{xy}\vert$ and $1 \leq \vert A_{xy}^M\vert = \vert B_{xy}^M \vert$ (when $p_1 \geq 2$). So, suppose that $E_1=E_{xy}\setminus \{e_{t_1},\ldots,e_{t_{p_1}}\} \neq \emptyset$, then there exists $P \in \mathfrak{P}$, such that $E(P)\cap E_1 \neq \emptyset$.

\noindent Since $P$ is not of the form $(x,a_1,\ldots,a_l,y)$, $l \geq 2$, and $E_1 \cap E(P) \neq \emptyset$, then $P=(\ldots, f_2,x,e_{s_1},\ldots,$ $e_{s_{p_2}},y,f_3,\ldots)$, where $\{f_2,f_3\} \subset E(G)$. We will consider two cases. 

If $p_2=1$, then $P=(\ldots,f_2,x,e_{s_1},y,f_3,\ldots)$. By the construction of $M$ we have that the edges $f(x,f_2)f(x,e_{s_1})$ and $f(y,e_{s_1})f(y,f_3)$ are in $M$. Thus, $f(x,f_2)f(x,e_{s_1}) \in M_x$, $f(y,e_{s_1})f(y,s_1) \in M_y$ and $e_{s_1} \in C_{xy}^M$.

On the other hand, if $p_2 \geq 2$, then $f(x,f_2)f(x,e_{s_1})$, $f(y,e_{s_{p_2}})f(y,f_3)$ and $f(y,e_{s_i})f(x,e_{s_{i+1}})$ are in $M$, for every $i \in \{1,\ldots,s_2-1\}$. Hence, $e_{s_1} \in A_{xy}^M$, $e_{s_{p_2}} \in B_{xy}^M$ and $e_{s_i} \in D_{xy}^M$, for every $i \in \{2,\ldots,p_2-1\}$.

If $E_2 = E_1\setminus \{e_{s_1},\ldots,e_{s_{p_2}}\}= \emptyset$, then  $C_{xy}^M=E_{xy}$ (when $p_1=p_2=1$) or $\vert C_{xy}^M\vert < \vert E_{xy}\vert$ and $1 \leq \vert A_{xy}^M\vert = \vert B_{xy}^M\vert$ (when $p_1 \geq 2$ or $p_2 \geq 2$). Otherwise, there exists $P' \in \mathfrak{P}$, such that $E_2 \cap E(P') \neq \emptyset$ and we can repeat this procedure and after a finite number of steps we obtain that $C_{xy}^M=E_{xy}$ or $\vert C_{xy}^M\vert < \vert E_{xy}\vert$ and $1 \leq \vert A_{xy}^M\vert = \vert B_{xy}^M\vert$.

Now, suppose that there exists $M$, a matching in $L_2^H(G)$, such that for every set $E_{xy}\neq \emptyset$, it holds that: $C_{xy}^M=E_{xy}$ or $\vert C_{xy}^M\vert < \vert E_{xy}\vert$ and $1 \leq \vert A_{xy}^M\vert = \vert B_{xy}^M\vert$.

If $M$ is a perfect matching in $L_2^H(G)\setminus M_J$, then there exists a partition of $E(G)$ into closed dynamic $H$-trails, by Theorem \ref{theo:2}. So, suppose that $M$ is not a perfect matching in $L_2^H(G)\setminus M_J$. Recall that $M_x=M\cap E(G_x)$, for every $x \in V(G)$.

Since $C_{xy}^M=E_{xy} \neq \emptyset$ or $1 \leq \vert A_{xy}^M\vert$, then for every $x \in V(G)$, $M_x \neq \emptyset$.

Let $N= \bigcup_{x\in V(G)} M_x$. Notice that $N$ is a matching in $L_2^H(G)\setminus M_J$, $A_{xy}^M = A_{xy}^N$, $B_{xy}^M = B_{xy}^N$, $C_{xy}^M = C_{xy}^N$ and $D_{xy}^M = D_{xy}^N$. So, for every set $E_{xy}\neq \emptyset$, it holds that: $C_{xy}^N = E_{xy}$ or $\vert C_{xy}^N\vert < \vert E_{xy}\vert$ and $1 \leq \vert A_{xy}^N\vert =\vert B_{xy}^N\vert$.

Let $x$ and $y$ be two different vertices in $V(G)$, such that $E_{xy}\neq \emptyset$. 

If $C_{xy}^N  = E_{xy}$, then $f(x,e)$ and $f(y,e)$ are $N$-saturate, for every $e \in E_{xy}$. So, we do not add any edge to $N$.

On the other hand, if $\vert C_{xy}^N\vert < \vert E_{xy}\vert $ and $1 \leq \vert A_{xy}^N\vert = \vert B_{xy}^N\vert$, then we will consider $G_{xy}$, the subgraph of $L_2^H(G)$ induced by the set $\{f(x,e) : e \in E_{xy}\setminus C_{xy}^N\} \cup \{f(y,e) : e \in E_{xy}\setminus C_{xy}^N\}$.

Since $\vert C_{xy}^N\vert < \vert E_{xy}\vert$, then $V(G_{xy}) \neq \emptyset$. Let $A_{xy}^N=\{e_1,\ldots, e_t\}$ and $B_{xy}^N=\{f_1,\ldots, f_t\}$, since $1 \leq \vert A_{xy}^N\vert = \vert B_{xy}^N\vert$ we have that $t \geq 1$.

If $D_{xy}^N  = \emptyset$, then we add the edges $f(y,e_i)f(x,f_i)$ to $N$, for every $i \in \{1, \ldots, t\}$. By the definition of the edges of $L_2^H(G)$ (point b) and the definition of $G_{xy}$, we have that $f(y,e_i)f(x,f_i) \in E(G_{xy})$. Moreover, $N\cup \{f(y,e_i)f(x,f_i) : i \in \{1, \ldots, t\} \}$ is still a matching.

On the other hand, when $D_{xy}^N  \neq \emptyset$, let $D_{xy}^N=\{g_1,\ldots,g_s\}$. We add the edges $f(y,e_1)f(x,g_1)$, $f(x,f_1)f(y,g_s)$, $f(y,e_i)f(x,f_i)$ and $f(y,g_j)f(x,g_{j+1})$ to $N$, for every $i \in \{2, \ldots, t\}$ and $j \in \{1,\ldots,s-1\}$.

By the definition of the edges of $L_2^H(G)$ and the definition of $G_{xy}$, we have that all the edges we add to $N$ are in $E(G_{xy})$ and $N$ is still a matching.

Thus, we add edges to $N$ in such a way every vertex in $G_{xy}$ is $N$-saturate.

So, we repeat the same process for $x$ and $y$ in $V(G)$ such that $E_{xy}\neq \emptyset$.

Therefore, $N$ is a perfect matching in $L_2^H(G)\setminus M_J$ and there exists a partition of $E(G)$ into closed dynamic $H$-trails, by Theorem \ref{theo:2} and, by the construction of $N$, none of them is of the form $(x,a_1,\ldots, a_l, y)$, with $l \geq 2$.
\end{proof}

The following result shows a condition on the graph $G_x$ that guarantees the existence of closed Euler dynamic $H$-trail.

\begin{theorem}\label{theo:5}
Let $G$ be a connected $H$-colored multigraph, such that $G_u$ is a complete $k_u$-partite graph for every $u \in V(G)$ and for some $k_u$ in $\mathbb{N}\setminus \{1\}$. Then $G$ has a closed Euler dynamic $H$-trail if and only if there exists a matching $M$ in $L_2^H(G)$, such that for every set $E_{xy} \neq \emptyset$ one of the following conditions hold: a) $C_{xy}^M=E_{xy}$; or b) $\vert C_{xy}^M\vert < \vert E_{xy}\vert$ and $1 \leq \vert A_{xy}^M\vert = \vert B_{xy}^M\vert$.
\end{theorem}
\begin{proof}
Let $G$ be a connected $H$-colored multigraph, such that $G_u$ is a complete $k_u$-partite graph for every $u \in V(G)$ and for some $k_u$ in $\mathbb{N}\setminus \{1\}$.

Suppose that $G$ has a closed Euler dynamic $H$-trail, say $P$.

\textbf{Case 1.} $P$ is not of the form $(x,e_1,\ldots,e_k,y)$, with $k \geq 2$.

In this case there exists a matching, say  $M$, in $L_2^H(G)$, such that for every set $E_{xy} \neq \emptyset$ one of the following conditions hold: $C_{xy}^M=E_{xy}$ or $\vert C_{xy}^M\vert < \vert E_{xy}\vert$ and $1 \leq \vert A_{xy}^M\vert = \vert B_{xy}^M\vert$.

\textbf{Case 2.} $P=(x,e_1,\ldots,e_k,y)$, with $k\geq 2$. 

In this case, since $G_x$ is a complete $k_x$-partite graph, for some $k_x \geq 2$, then there exists two edges $e$ and $f$ in $E(G)=E(P)$, such that $c(e)c(f) \in E(H)$. Suppose, without loss of generality, that $c(e_1)c(e_2) \in E(H)$.

If $k=2$, then $P'=(x,e_1,y,e_2,x)$ is a closed Euler dynamic $H$-trail. Otherwise, since $G_x$ is $k_x$-partite graph, we have that $c(e_1)c(e_k) \in E(H)$ or $c(e_2)c(e_k) \in E(H)$. 

Hence, $P'=(x,e_1,y,e_2,\ldots,e_k,x)$ or $P'=(x,e_1,e_3,\ldots,e_k,y,e_2,x)$ is a closed Euler dynamic $H$-trail. By Theorem \ref{theo:4}, there exists a matching $M$ in $L_2^H(G)$, such that for every set $E_{xy} \neq \emptyset$, $C_{xy}^M=E_{xy}$ or $\vert C_{xy}^M\vert < \vert E_{xy}\vert$ and $1 \leq \vert A_{xy}^M\vert = \vert B_{xy}^M\vert$.\vspace{1em}

Conversely, suppose that there exists a matching $M$ in $L_2^H(G)$, such that for every set $E_{xy} \neq \emptyset$ one of the following conditions hold: $C_{xy}^M=E_{xy}$ or $\vert C_{xy}^M\vert < \vert E_{xy} \vert$ and $1 \leq \vert A_{xy}^M\vert = \vert B_{xy}^M\vert$.

It follows from Theorem \ref{theo:4} that there exists a partition of $E(G)$ into closed dynamic $H$-trails, none of them in the form $(x,e_1,\ldots,e_l,y)$, with $l \geq 2$, say $\mathcal{P}=\{P_1,\ldots,P_k\}$.

If $k=1$, then $E(P_1)=E(G)$ and $P_1$ is a closed Euler dynamic $H$-trail. Otherwise, since $G$ is connected it follows that there exists $P_i \in \mathcal{P}\setminus \{P_1\}$, say $P_2$, such that $V(P_1)\cap V(P_2)\neq \emptyset$. Since $P_1$ and $P_2$ are closed, we can suppose that $P_1=(x_0,e_1,\ldots,e_m,x_0)$, where $e_1=x_0x_1$ and $e_m=x_{k_1}x_0$, and $P_2=(y_0,f_1,\ldots,f_n,y_0)$, where $f_1=y_0y_1$ and $f_n=y_{k_2}y_0$, such that $x_0=y_0$. By the definition of dynamic $H$-trail, we have that $c(e_1)c(e_m)$ and $c(f_1)c(f_n)$ are in $E(H)$. Furthermore, $e_1,f_1,e_m$ and $f_n$ are incident with $x_0=y_0$. Hence, $e_1e_m$ and $f_1 f_n$ are in $E(G_{x_0})$.

We will consider the following: if $\{e_1,f_1\} \subseteq E_j^{x_0}$ or $\{e_m, f_n\} \subseteq E_j^{x_0}$, for some $j \in \{1,\ldots,k_{x_0}\}$, it follows that $e_1 f_n$ and $e_m f_1$ are in $E(G_{x_0})$, because $G_{x_0}$ is a complete $k_{x_0}$-partite graph. In other case, $e_1 f_1$ and $e_m f_n$ are in $E(G_{x_0})$.

By the definition of $G_{x_0}$, we have that $c(e_1)c(f_n)$ and $c(e_m)c(f_1)$ are edges of $H$ or $c(e_1)c(f_1)$ and $c(e_m)c(f_n)$ are edges of $H$.

If $c(e_1)c(f_n)$ and $c(e_m)c(f_1)$ are edges of $H$, we have that $Q_1=P_1 \cup P_2$ is a closed dynamic $H$-trail. And, if $c(e_1)c(f_1)$ and $c(e_m)c(f_n)$ are edges of $H$, we have that $Q_1=P_1 \cup P_2^{-1}$ is a closed dynamic $H$-trail.

If $E(Q_1)=E(G)$, then $Q_1$ is a closed Euler dynamic $H$-trail. Otherwise, since $G$ is connected it follows that there exists $P_i \in \mathcal{P}\setminus \{P_1,P_2\}$, say $P_3$, such that $V(Q_1)\cap V(P_3)\neq \emptyset$. Since $Q_1$ and $P_3$ are closed, we can suppose that $P_1=(v_1,a_1,\ldots,a_s,v_0)$, where $a_1=v_0v_1$ and $a_s=v_{k_3}v_0$, and $P_3=(u_0,g_1,\ldots,g_t,u_0)$, where $g_1=u_0u_1$ and $g_t=u_{k_3}u_0$, such that $v_0=u_0$. By the definition of dynamic $H$-trial, we have that $c(a_1)c(a_s)$ and $c(g_1)c(g_t)$ are in $E(H)$. Furthermore, $a_1,g_1,a_s$ and $g_t$ are incident with $u_0=v_0$. Hence, $a_1a_s$ and $g_1g_t$ are in $E(G_{v_0})$.

We will consider the following: if $\{a_1,g_1\} \subseteq E_j^{v_0}$ or $\{a_s, g_t\} \subseteq E_j^{v_0}$, for some $j \in \{1,\ldots,k_{v_0}\}$, it follows that $a_1 g_t$ and $a_s g_1$ are in $E(G_{v_0})$, because $G_{v_0}$ is a complete $k_{v_0}$-partite graph. In other case, $a_1 g_1$ and $a_s g_t$ are in $E(G_{v_0})$.

By the definition of $G_{v_0}$, we have that $c(a_1)c(g_t)$ and $c(a_s)c(g_1)$ are edges of $H$ or $c(a_1)c(g_1)$ and $c(a_s)c(g_t)$ are edges of $H$.

If $c(a_1)c(g_t)$ and $c(a_s)c(g_1)$ are edges of $H$, we have that $Q_2=Q_1 \cup P_3$ is a closed dynamic $H$-trail. And, if $c(a_1)c(g_1)$ and $c(a_s)c(g_t)$ are edges of $H$, we have that $Q_2=Q_1 \cup P_3^{-1}$ is a closed dynamic $H$-trail.

If $E(Q_2)=E(G)$, then $Q_2$ is a closed Euler dynamic $H$-trail. Otherwise, we repeat the procedure until we get the desired closed Euler dynamic $H$-trail.
\end{proof}

In Theorem \ref{theo:5} we cannot change the hypothesis ``$G_u$ is a complete $k_u$-partite for every $u \in V(G)$" to ``$G_u$ has a complete $k_u$-partite spanning subgraph", as Figure \ref{fig:Conter2} shows.

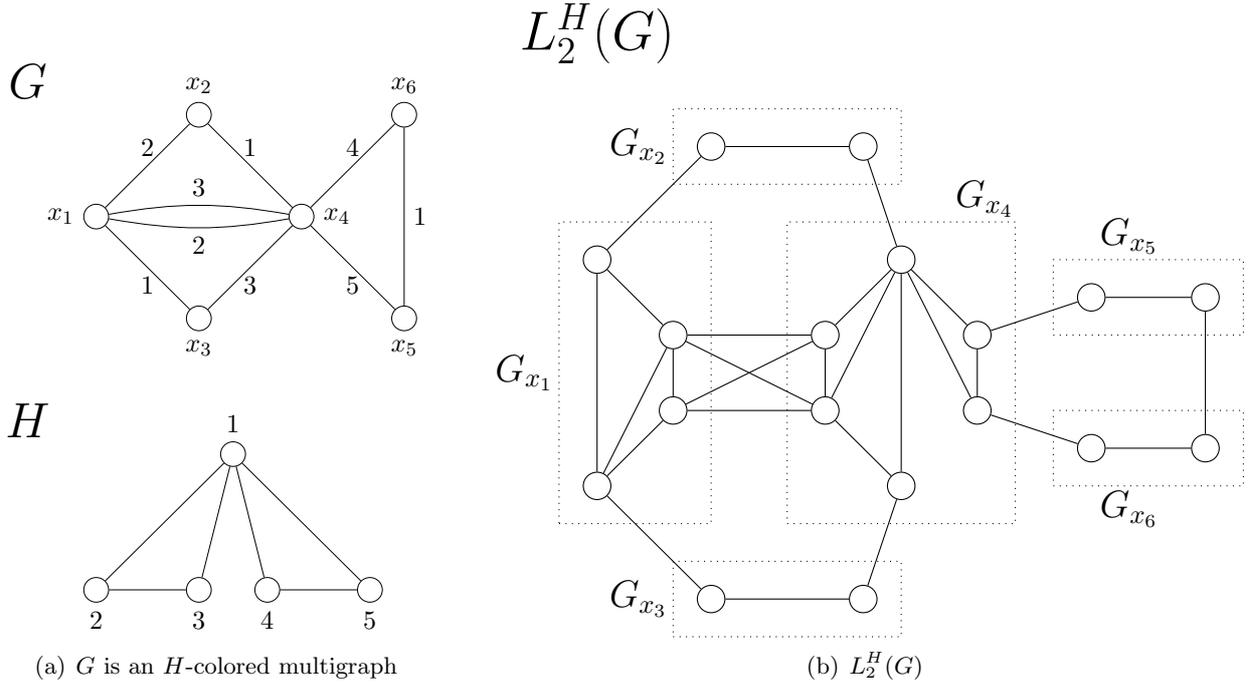
\begin{figure}[tpb]
\centering 
	\subfigure[$G$ is an $H$-colored multigraph]
	{\scalebox{0.9}{\begin{tikzpicture}
	\tikzset{every loop/.style={min distance=15mm,in=120,out=60,looseness=10}}
	\tikzstyle{vertex}=[circle, draw=black] 
%%%%%%%%Gráfica G
		\node[](G) at (0,3.5){\huge{$G$}};		
		\node[vertex](v1) at (1,1.5) {};
		\node[left=0mm of v1](Et.u) {$x_1$};%etiqueta de v1
		\node[vertex](v2) at (2.5,3) {};
		\node[above=0mm of v2](Et.u) {$x_2$};%etiqueta de v2
		\node[vertex](v3) at (2.5,0) {};
		\node[below=0mm of v3](Et.u) {$x_3$};%etiqueta de v3
		\node[vertex](v4) at (4,1.5) {};
		\node[right=0mm of v4](Et.u) {$x_4$};%etiqueta de v4
		\node[vertex](v5) at (5.5,0) {};
		\node[below=0mm of v5](Et.u) {$x_5$};%etiqueta de v5
		\node[vertex](v6) at (5.5,3) {};
		\node[above=0mm of v6](Et.u) {$x_6$};%etiqueta de v6
		
		\draw[] (v1) -- node[above] {2} ++(v2);
		\draw[] (v1) -- node[below] {1} ++(v3);
		\draw[] (v2) -- node[above] {1} ++(v4);
		\draw[] (v3) -- node[below] {3} ++(v4);
		\draw[] (v4) -- node[below] {5} ++(v5);
		\draw[] (v4) -- node[above] {4} ++(v6);
		\draw[] (v5) -- node[right] {1} ++(v6);
		\path (v1) 	edge [above,bend left=10] node {3} (v4)
					edge [below,bend right=10] node {2} (v4);
		
%%%%%%%%Gráfica H
		\node[](H) at (0,-1.5){\huge{$H$}};
		\node[vertex](v1) at (1,-4) {};
		\node[below=0mm of v1](Et.u) {2};%etiqueta de v1
		\node[vertex](v2) at (2.5,-4) {};
		\node[below=0mm of v2](Et.u) {3};%etiqueta de v2
		\node[vertex](v3) at (3.5,-4) {};
		\node[below=0mm of v3](Et.u) {4};%etiqueta de v3
		\node[vertex](v4) at (5,-4) {};
		\node[below=0mm of v4](Et.u) {5};%etiqueta de v4
		\node[vertex](v5) at (3,-2) {};
		\node[above=0mm of v5](Et.u) {1};%etiqueta de v5
		
		\draw[] (v5) -- (v1);
		\draw[] (v5) -- (v2);
		\draw[] (v5) -- (v3);
		\draw[] (v5) -- (v4);
		\draw[] (v1) -- (v2);
		\draw[] (v3) -- (v4);
		
	\end{tikzpicture}}}
	\hfill
%%%%%%%%%%%%Gráfica Ggv	
	\subfigure[$L_2^H(G)$]{\scalebox{1}{\begin{tikzpicture}
	\tikzset{every loop/.style={min distance=15mm,in=120,out=60,looseness=10}}
	\tikzstyle{vertex}=[circle, draw=black] 
		
		\node[](L2) at (0,7){\huge{$L_2^H(G)$}};
		\node[vertex](e1x1) at (0,4) {};
%		\node[above=0mm of e1v1](Et.u) {$v_1$};%etiqueta de e1
		\node[vertex](e1x2) at (1.5,5.5) {};
%		\node[above=0mm of e1v2](Et.u) {$v_2$};%etiqueta de e1
		\node[vertex](e2x1) at (1,3) {};
%		\node[above=0mm of e2v1](Et.u) {$v_3$};%etiqueta de e2
		\node[vertex](e2x4) at (3,3) {};
%		\node[below=0mm of e2v4](Et.u) {$v_4$};%etiqueta de e2
		\node[vertex](e3x1) at (1,2) {};
%		\node[below=0mm of e3v4](Et.u) {$v_3$};%etiqueta de e3
		\node[vertex](e3x4) at (3,2) {};
%		\node[below=0mm of e3v5](Et.u) {$v_4$};%etiqueta de e3
		\node[vertex](e4x1) at (0,1) {};
%		\node[above=0mm of e4v2](Et.u) {$v_1$};%etiqueta de e4
		\node[vertex](e4x3) at (1.5,-0.5) {};
%		\node[above=0mm of e4v5](Et.u) {$v_2$};%etiqueta de e4
		\node[vertex](e5x3) at (3.5,-0.5) {};
%		\node[above=0mm of e5v2](Et.u) {$v_3$};%etiqueta de e5
		\node[vertex](e5x4) at (4,1) {};
%		\node[below=0mm of e5v5](Et.u) {$v_4$};%etiqueta de e5
		\node[vertex](e6x2) at (3.5,5.5) {};
%		\node[below=0mm of e6v2](Et.u) {$v_3$};%etiqueta de e6
		\node[vertex](e6x4) at (4,4) {};
%		\node[below=0mm of e6v3](Et.u) {$v_4$};%etiqueta de e6
		\node[vertex](e7x4) at (5,3) {};
%		\node[below=0mm of e7v3](Et.u) {$v_3$};%etiqueta de e7
		\node[vertex](e7x5) at (6.5,3.5) {};
%		\node[below=0mm of e7v6](Et.u) {$v_4$};%etiqueta de e7
		\node[vertex](e8x4) at (5,2) {};
%		\node[below=0mm of e8v5](Et.u) {$v_3$};%etiqueta de e8
		\node[vertex](e8x6) at (6.5,1.5) {};
%		\node[below=0mm of e8v6](Et.u) {$v_4$};%etiqueta de e8
		\node[vertex](e9x5) at (8,3.5) {};
%		\node[below=0mm of e8v5](Et.u) {$v_3$};%etiqueta de e9
		\node[vertex](e9x6) at (8,1.5) {};
%		\node[below=0mm of e8v5](Et.u) {$v_3$};%etiqueta de e9

		\draw[] (e1x1) -- (e1x2);
		\draw[] (e2x1) -- (e2x4);
		\draw[] (e2x1) -- (e3x4);
		\draw[] (e2x4) -- (e3x1);
		\draw[] (e3x1) -- (e3x4);
		\draw[] (e4x1) -- (e4x3);
		\draw[] (e5x3) -- (e5x4);
		\draw[] (e6x2) -- (e6x4);
		\draw[] (e7x4) -- (e7x5);
		\draw[] (e8x4) -- (e8x6);
		\draw[] (e9x5) -- (e9x6);
		\draw[] (e1x1) -- (e2x1);
		\draw[] (e1x1) -- (e4x1);
		\draw[] (e2x1) -- (e3x1);
		\draw[] (e2x1) -- (e4x1);
		\draw[] (e3x1) -- (e4x1);
		\draw[] (e1x2) -- (e6x2);
		\draw[] (e4x3) -- (e5x3);
		\draw[] (e6x4) -- (e2x4);
		\draw[] (e6x4) -- (e3x4);
		\draw[] (e6x4) -- (e5x4);
		\draw[] (e6x4) -- (e7x4);
		\draw[] (e6x4) -- (e8x4);
		\draw[] (e2x4) -- (e3x4);
		\draw[] (e3x4) -- (e5x4);
		\draw[] (e7x4) -- (e8x4);
		\draw[] (e7x5) -- (e9x5);
		\draw[] (e8x6) -- (e9x6);
		
		\draw[dotted, thin] (-0.5,0.5) rectangle (1.5,4.5);
		\node[](G1) at (-0.3,2.5){};	
		\node[left=0mm of G1] {\Large{$G_{x_1}$}};
		\draw[dotted, thin] (1,5) rectangle (4,6);
		\node[](G2) at (1.2,5.5){};	
		\node[left=0mm of G2] {\Large{$G_{x_2}$}};
		\draw[dotted, thin] (1,-1) rectangle (4,0);
		\node[](G3) at (1.2,-0.5){};	
		\node[left=0mm of G3] {\Large{$G_{x_3}$}};
		\draw[dotted, thin] (2.5,0.5) rectangle (5.5,4.5);
		\node[](G4) at (5.1,4.3){};	
		\node[above=0mm of G4] {\Large{$G_{x_4}$}};
		\draw[dotted, thin] (6,3) rectangle (8.5,4);
		\node[](G5) at (7,3.8){};	
		\node[above=0mm of G5] {\Large{$G_{x_5}$}};
		\draw[dotted, thin] (6,1) rectangle (8.5,2);
		\node[](G6) at (7,1.2){};	
		\node[below=0mm of G6] {\Large{$G_{x_6}$}};
		
	\end{tikzpicture}}}
	\caption{$G_u$ is a complete $k_u$-partite graph, for $u \in V(G)\setminus \{x_4\}$, and $G_{x_4}$ has a bipartite spanning subgraph but there is no closed Euler dynamic $H$-trail in $G$}
\label{fig:Conter2}
\end{figure}

\begin{corollary}\label{cor:1}
Let $G$ be an $H$-colored multigraph. Then $G_x$ has a perfect matching, for every $x \in V(G)$, if and only if there exists a partition of $E(G)$ into closed $H$-trails. 
\end{corollary}
\begin{proof}
Let $G$ be an $H$-colored multigraph and $M_J$ be the joint matching of $L_2^H(G)$. 

Suppose that $G_x$ has a perfect matching, for every $x \in V(G)$, say $M_x$. Let $M=\bigcup_{x\in V(G)} M_x$.

It follows from the definition of $M_J$ that $M \cap M_J =\emptyset$. Hence, $G$ has a partition of its edges into closed dynamic $H$-trails, by Theorem \ref{theo:2}.

Let $\mathcal{P}$ be the partition that is obtained as in the proof of Theorem \ref{theo:2}. It follows, from the construction of each dynamic $H$-trail in $\mathcal{P}$, that there is no lane change. So, $\mathcal{P}$ is a partition of $E(G)$ into closed $H$-trails.

Now, suppose that there exists a partition of $E(G)$ into closed $H$-trails, say $\mathcal{P}$. By Theorem \ref{theo:2}, there exists a perfect matching in $L_2^H(G)\setminus M_J$. Let $M$ be the perfect matching that is obtained as in the proof of Theorem \ref{theo:2}.

Let $f(x,e) \in V(L_2^H(G))$. Then, there exists $P \in \mathcal{P}$ such that $e \in E(P)$. Since $P$ is an $H$-trail and by the construction of $M$, it follows that $P=(u,g,x,e,v,\ldots)$ and $f(x,g)f(x,e) \in M$. Moreover, $f(x,g)f(x,e) \in M \cap E(G_x)$.

Therefore, $M \cap E(G_x)$ is a perfect matching of $G_x$.
\end{proof}

Notice that if $H$ is a complete graph without loops and $G$ is an $H$-colored multigraph. Then $P$ is a properly colored closed trail if and only if $P$ is an closed $H$-trail.

\begin{corollary}[Kotzig]
Let $G$ be an edge-colored eulerian multigraph. Then $G$ has a properly colored closed Euler trail if and only if $\delta_i(x) \leq \sum_{j\neq i} \delta_j (x)$, where $\delta_i(x)$ is the number of edges with color $i$ incident with $x$, for each vertex $x$ of $G$.
\end{corollary}
\begin{proof}
Let $G$ be an $H$-colored multigraph, where $H$ is a complete graph without loops and $\vert V(H)\vert =c$.

Let $G_u$, with $u \in V(G)$. It follows from the definition of $G_u$ and $H$ is a complete graph that $G_u$ is a complete $k_u$-paritite graph with independent sets $E_i^u=\{e \in E(G)\: : \: e=ux \textrm{ for some } x \in V(G) \textrm{ and } c(e)=i\}$, with $i \in \{1,\ldots,c\}$. Moreover, $\vert E_i^u\vert = \delta_i(u)$.

Suppose that $G$ has a closed Euler $H$-trail. By Corollary \ref{cor:1} $G_u$ has a perfect matching, for every $u \in V(G)$.

Hence, $\delta_i(u)=\vert E_i^u\vert = o(G \setminus S) \leq \vert S\vert = \sum_{j\neq i} \vert E_j^u\vert = \sum_{j\neq i} \delta_i(u)$, where $S= V(G_u)\setminus E_i^u$, by Theorem \ref{theo:Tutte}.

Suppose that $\vert E_i^u\vert =\delta_i(x) \leq \sum_{j\neq i} \delta_j (x)= \sum_{j\neq i} \vert E_j^u\vert$, for each vertex $x$ of $G$.

\textbf{Claim 1.} $G_u$ has a perfect matching.

Let $S$ be a proper subset of $V(G_u)$. Consider $G_u \setminus S$.

Case 1. $G_u \setminus S$ is connected.

Then $o(G_u \setminus S) \leq 1 \leq \vert S\vert$.\vspace{1em}

Case 2. $G_u \setminus S$ is disconnected.

Since $G_u$ is a complete $k_u$-partite graph, then $V(G_u \setminus S) \subseteq E_i^u$, for some $i \in \{1,\ldots,c\}$. Hence $o(G_u \setminus S) = \vert V(G_u \setminus S)\vert \leq \vert E_i^u\vert \leq \sum_{j\neq i} \vert E_j^u\vert \leq \vert S \vert$.

Therefore, it follows from Theorem \ref{theo:Tutte} that $G_u$ has a perfect matching.

Let $M=\bigcup_{u\in V(G)} M_u$, where $M_u$ is a perfect matching in $G_u$. So, $M$ is a perfect matching in $L_2^H(G)$ and $C_{xy}^M=E_{xy}$, for every $E_{xy} \neq \emptyset$.

It follows from Theorem \ref{theo:5} and Corollary \ref{cor:1} that $G$ has a closed Euler $H$-trail.
\end{proof}

\begin{corollary}[\cite{H-paseos}]
Let $H$ be a graph possibly with loops and $G$ be an $H$-colored multigraph without loops. Suppose that $G$ is Eulerian and $G_u$ is a complete $k_u$-partite graph, for every $u$ in $V(G)$ and for some $k_u$ in $\mathbb{N}$. Then $G$ has a closed Euler $H$-trail if and only if $\vert C_i^u \vert\leq \sum_{j \neq i} \vert C_j^u\vert$ for every $u$ in $V(G)$, where $\{C_1^u,...,C_{k_u}^u\}$ is the partition of $V(G_u)$ into independent sets.
\end{corollary}

\end{document}